\documentclass[a4paper, 11pt]{article}

\usepackage[utf8]{inputenc}
\usepackage[OT1]{fontenc}
\usepackage{lmodern}     
\usepackage[english]{babel}
\usepackage{cite, hyphenat, hyperref}
\usepackage{amscd, amsfonts, amsmath}
\usepackage{amssymb, amsthm}
\usepackage{bbm, mathrsfs}
\usepackage{enumerate}
\usepackage{makeidx}
\usepackage{calc}
\usepackage[shortlabels]{enumitem}
\usepackage{nccfoots}
\usepackage{appendix}
\usepackage[a4paper, top=2.5cm, bottom=2.5cm, left=3cm, right=3cm]{geometry}
\allowdisplaybreaks
\numberwithin{equation}{section}
\theoremstyle{plain}

\newtheorem{Lemma}{Lemma}[section]
\newtheorem{Proposition}[Lemma]{Proposition}

\newtheorem{Corollary}[Lemma]{Corollary}

\theoremstyle{definition}
\newtheorem{Definition}[Lemma]{Definition}
\newtheorem{Example}[Lemma]{Example}

\newtheorem{Remark}[Lemma]{Remark}

\def\R{\mathbb{R}}
\def\N{\mathbb{N}}
\begin{document}
%-------------------------------------------------------------------
% Title
%-------------------------------------------------------------------
\title{\textbf{The oriented derivative}}
\author{Alexander Kalinin\footnote{Department of Mathematics, LMU Munich, Germany. {\tt alex.kalinin@mail.de}}}
\maketitle

\begin{abstract}
We show that the derivatives in the sense of Fr{\'e}chet and G{\^a}teaux can be viewed as derivatives oriented towards a star convex set with the origin as center. The resulting oriented differential calculus extends the mean value theorem, the chain rule and the Taylor formula in Banach spaces. Moreover, the oriented derivative decomposes additively along countably infinite orthogonal sums in Hilbert spaces.
\end{abstract}


\section{A unifying differentiability notion}

The derivative of a function is a central object in calculus that carries intrinsic properties of the function itself, as the mean value theorem and the fundamental theorem of calculus entail. Indeed, the analysis of a function, defined on a domain of even infinite dimension, by means of its derivative is an essential part of many textbooks on calculus, such as~\cite{Die60, Wal76, AmaEsc05, AmaEsc08}, and it goes back as far as the classical works of Fr{\'e}chet~\cite{Fre11} and G{\^a}teaux~\cite{Gat13, Gat19}.

In this paper, we present a differential calculus that applies to the derivatives of Fr{\'e}chet and G{\^a}teaux and solely requires the function and its domain to be differentiable and open in the direction of the selected orientation, respectively. In particular, in a Hilbert space $X$ the oriented derivative can be identified with a gradient.

Namely, for a $0$-star convex set $S$ in $X$ let $U$ be a non-empty $S$-open set in $X$ in the sense that for any $x\in U$ there is $\delta > 0$ such that $x + h\in U$ for all $h\in S$ with $|h| < \delta$. Then $\varphi:U\rightarrow\R$ is \emph{$S$-differentiable} if there is a map $L:U\rightarrow X$ satisfying
\begin{equation}\label{eq:special oriented gradient}
\varphi(x + h) - \varphi(x) = \langle L(x),h\rangle + o(|h|)\quad\text{as $h\rightarrow 0$ on $S$}
\end{equation}
for any $x\in U$, where $\langle\cdot,\cdot\rangle$ is the inner product on $X$ that induces the complete norm $|\cdot|$. In this case, there exists a unique map $\nabla_{S}\varphi$ on $U$ taking all its values in the closure $V$ of the linear hull of $S$ such that
\begin{equation*}
\langle L,h\rangle = \langle\nabla_{S}\varphi,h\rangle\quad\text{for all $h\in V$.}
\end{equation*}
We shall call $\nabla_{S}\varphi$ the \emph{$S$-oriented gradient} of $\varphi$ and let $C_{S}^{1}(U)$ stand for the the linear space of all $S$-differentiable functions $\varphi\in C(U)$ for which $\nabla_{S}\varphi$ is continuous. Several relevant applications and properties of the oriented gradient are as follows.\smallskip

(i) While $\varphi:U\rightarrow\R$ is $X$-differentiable if and only if it is \emph{Fr{\'e}chet} differentiable, the introduced concept in~\eqref{eq:special oriented gradient} is redundant for $S = \{0\}$. Thus, $C_{X}^{1}(U) = C^{1}(U)$ and $C_{\{0\}}^{1}(U) = C(U)$ and we recover the \emph{gradient} and the zero operator,
\begin{equation*}
\nabla_{X} = \nabla \quad\text{and}\quad \nabla_{\{0\}} = 0.
\end{equation*}

(ii) If the linear hull of $S$ is dense in $X$, then $\nabla_{S}$ determines $\nabla$ in the following sense: Every differentiable extension $\tilde{\varphi}$ of $\varphi$ to an open set $\tilde{U}$ in $X$ with $U\subset\tilde{U}$ satisfies
\begin{equation}\label{eq:derivative of an extension}
\nabla\tilde{\varphi} = \nabla_{S}\varphi.
\end{equation}
In the case that $X = \R^{d}$ for some $d\in\N$ and $S$ is the half-space $\mathbb{H}^{d}$ of all $x\in\R^{d}$ with $x_{d}\geq 0$, this fact is used to introduce differentiable manifolds with boundary. For instance, see~\cite{Mun91, AmaEsc09}. In particular, for $d = 1$ the function $\varphi$ is $\R_{+}$-differentiable if and only if its \emph{right-hand derivative} $\varphi_{+}'(x)$ exists at any $x\in U$. In this case, $\nabla_{\R_{+}}\varphi = \varphi_{+}'$.

(iii) Let $S$ be balanced and its linear hull be dense in $X$. If $U$ is convex and $\varphi\in C_{S}^{1}(U)$, then, regardless of whether $\varphi$ is actually differentiable in the interior of $U$, the \emph{mean value theorem for the oriented derivative} in Corollary~\ref{co:mean value} yields that
\begin{equation*}
\varphi(x) - \varphi(y) = \int_{0}^{1}\langle\nabla_{S}\varphi((1-t)x + ty),x - y\rangle\,dt\quad\text{for any $x,y\in U^{\circ}$.}
\end{equation*}
For example, let $X$ be the separable Banach space $C([0,T],\R^{d})$ of all $\R^{d}$-valued continuous maps on $[0,T]$ for $T > 0$, endowed with the supremum norm. Then for $S$ we may take any linear space that contains all piecewise affine maps, such as the Cameron-Martin space $H^{1}([0,T],\R^{d})$ of all absolutely continuous maps with a square-integrable weak derivative that plays a major role in~\cite{ConKal20} in the context of stochastic and functional It{\^o} calculus.

(iv) The \emph{directional} or \emph{G{\^a}teaux derivative} of $\varphi$ at $x\in U$ \emph{in the positive direction} of $h\in S$, defined as $D_{h}^{+}\varphi(x) = \lim_{t\downarrow 0} (\varphi(x + th) - \varphi(x))/t$, exists if and only if $\varphi$ is differentiable at $x$ relative to the conic hull of $\{h\}$. In this case,
\begin{equation*}
D_{h}^{+}\varphi(x) = \langle\nabla_{\mathrm{coni}(h)}\varphi(x),h\rangle.
\end{equation*}
Hence, if $U$ is convex, then $h = y -x$ with $y\in U$ is possible. Further, positive homogeneity of $S$ and $S$-differentiability of $\varphi$ at $x$ imply that $D_{h}^{+}\varphi(x) = \langle\nabla_{S}\varphi(x),h\rangle$ for all $h\in S$.

(v) Let $S$ be balanced and for $n\in\N$ let $S_{1},\dots,S_{n}$ be pairwise orthogonal balanced sets in $X$ such that $S = S_{1}\oplus\cdots\oplus S_{n}$. Then $C_{S}^{1}(U)$ $= C_{S_{1}}^{1}(U)\cap\cdots\cap C_{S_{n}}^{1}(U)$ and we obtain the \emph{orthogonal decomposition}
\begin{equation}\label{eq:orthogonal decomposition of the gradient}
\nabla_{S}\varphi = \nabla_{S_{1}}\varphi + \cdots + \nabla_{S_{n}}\varphi\quad\text{for any $\varphi\in C_{S}^{1}(U)$,}
\end{equation}
as the explicit description of $C_{S}^{1}(U)$ in Corollary~\ref{co:decomposition of the derivative} shows. In particular, let $X = \R^{d}$, $d_{1},\dots,d_{n}\in\N$ satisfy $d = d_{1} + \cdots + d_{n}$ and $T_{i}$ be a balanced set in $\R^{d_{i}}$ for all $i\in\{1,\dots,n\}$ such that
\begin{equation}\label{eq:Cartesian product}
S = T_{1}\times\cdots\times T_{n}.
\end{equation}
We note that for $n=1$ this forces $S_{1} = T_{1}$. If instead $n\geq 2$, then for~\eqref{eq:Cartesian product} to hold it suffices to assume that $S_{1} = T_{1}\times\{0\}\times\cdots\times\{0\}$, $S_{n} = \{0\}\times\cdots\times\{0\}\times T_{n}$ and
\begin{equation*}
S_{i} = \{0\}\times\cdots\times\{0\}\times T_{i}\times\{0\}\times\cdots\times\{0\}\quad\text{for any $i\in\{2,\dots,n-1\}$.}
\end{equation*}
So, we write each $x\in\R^{d}$ in the form $x = (x_{1},\dots,x_{n})$ with $x_{1}\in\R^{d_{1}},\dots,x_{n}\in\R^{d_{n}}$ and let $\nabla_{T_{i}}$ denote the $i$-th coordinate of the differential operator $\nabla_{S_{i}}$ for all $i\in\{1,\dots,n\}$. Then~\eqref{eq:orthogonal decomposition of the gradient} turns into
\begin{equation*}
\nabla_{S}\varphi =
\begin{pmatrix}
\nabla_{T_{1}}\varphi\\
\vdots\\
\nabla_{T_{n}}\varphi
\end{pmatrix}
\quad\text{for all $\varphi\in C_{S}^{1}(U)$}
\end{equation*}
and in the particular case that $T_{i} = \R^{d_{i}}$ for $i\in\{1,\dots,n\}$ we recover the transpose of the \emph{partial derivative} with respect to the coordinate $x_{i}\in\R^{d_{i}}$. Namely,
\begin{equation*}
\nabla_{\R^{d_{i}}} = \nabla_{x_{i}} = D_{x_{i}}'.
\end{equation*}
Finally, in the specific scenario~\eqref{eq:Cartesian product} the oriented derivative $\nabla_{S}$ yields a unifying theory for classical and viscosity solutions to elliptic and parabolic equations, as the companion paper~\cite{Kal23} will show.\smallskip

To cover these applications and derive the emphasized results, we develop a \emph{differential calculus for the oriented derivative} from its classical counterpart in Section~\ref{se:2}. Based on a differentiability concept for Banach spaces, we deduce the chain rule, the mean value theorem and a characterisation via projections in Section~\ref{se:2.1}.

The orthogonal decomposition of the oriented derivative in Hilbert spaces and the definition of the oriented gradient, which yield the description~\eqref{eq:special oriented gradient} and the identity~\eqref{eq:orthogonal decomposition of the gradient}, are discussed in Section~\ref{se:2.2}.

Higher order derivatives with the Hessian as special case, Taylor's formula and relations between local extrema and first- and second-order oriented derivatives are considered in Section~\ref{se:2.3}. All the results are proven in Section~\ref{se:3}.

\section{Oriented differential calculus}\label{se:2}

In the sequel, let $X$ and $Y$ be Banach spaces, $|\cdot|$ stand for the complete norm on both $X$ and $Y$ and $S$ be a star-convex set in $X$ with center $0$. We write $V$ for the closure of the linear hull $\mathrm{span}(S)$ of $S$ and let $U$ be a non-empty set in $X$.

By $B_{r}(x)$ we denote the open ball around a point $x$ in one of the Banach spaces with radius $r\geq 0$, recall the conic hull $\mathrm{coni}(S)$ of $S$ and set $[x,y]:=\{(1 - t)x + t y\,|\,t\in [0,1]\}$ for all $x,y\in X$. Further, $\mathcal{L}(X,Y)$ is the Banach space of all $Y$-valued linear continuous operators on $X$, equipped with the operator norm that is also denoted by $|\cdot|$.

\subsection{Differentiability, the chain rule and the mean value theorem}\label{se:2.1}

We introduce and characterise the oriented derivative that includes the gradient in~\eqref{eq:special oriented gradient} as special case and allows for extended versions of the chain rule and the mean value theorem.

To this end, let us define the \emph{$S$-interior} of $U$ to be the set $U_{S}^{\circ}$ of all $x\in U$ such that $x + B_{\delta}(0)\cap S \subset U$ for some $\delta > 0$, which includes the standard interior of $U$. Accordingly, $U$ is called \emph{$S$-open} if it coincides with $U_{S}^{\circ}$.

\begin{Example}\label{ex:interior points}
Let $X = \R^{d}$ and $S=[0,R[^{d}$ (resp.~$S=]-R,0]^{d}$) for $d\in\N$ and $R\in ]0,\infty]$. Then $x\in U_{S}^{\circ}$ if and only if there is $\delta > 0$ such that
\begin{equation*}
[x_{1},x_{1} + \delta[\times\cdots\times [x_{d},x_{d} + \delta[\,\subset U\quad\text{(resp.~$]x_{1} - \delta,x_{1}]\times\cdots\times ]x_{d} - \delta,x_{d}]\subset U$).}
\end{equation*}
In particular, $\R_{+}^{d}$ is $[0,R[^{d}$-open and $]-\infty,0]^{d}$ is $]-R,0]^{d}$-open.
\end{Example}

We notice that a map $L:S\rightarrow Y$ satisfying $L(th) = t L(h)$ for all $t\in [0,1]$ and $h\in S$ and $L(h) = o(|h|)$ as $h\rightarrow 0$ must vanish, by the positive homogeneity of a norm. This follows follows readily from the identities
\begin{equation*}
\frac{L(h)}{|h|} = \lim_{n\uparrow\infty} \frac{L(\frac{h}{n})}{|\frac{h}{n}|} = 0\quad\text{for any $h\in S\setminus\{0\}$.}
\end{equation*}
If in fact $L$ is the restriction of some $\tilde{L}\in\mathcal{L}(V,Y)$ to $S$, then $\tilde{L} = 0$, as $V$ is the closure of $\mathrm{span}(S)$. This fact uniquely determines the derivative relative to $S$ declared as follows.

\begin{Definition}\label{de:differentiability}
A map $\varphi:U\rightarrow Y$ is said to be \emph{$S$-differentiable} at a point $x\in U_{S}^{\circ}$ if there is an operator $L\in\mathcal{L}(V,Y)$ such that
\begin{equation*}
\varphi(x + h) = \varphi(x) + L(h) + o(|h|)\quad\text{as $h\rightarrow 0$ on $S$.}
\end{equation*}
In this case, $L$ is called the \emph{$S$-oriented derivative} or simply the \emph{$S$-derivative} of $\varphi$ at $x$ and is denoted by $D_{S}\varphi(x)$.
\end{Definition}

\begin{Remark}\label{re:differentiability}
The $S$-derivative can be viewed as directional derivative in the positive direction of all the vectors in $S$, since $S = \{ t h \,|\,(t,h)\in [0,1]\times S\}$. In particular,
\begin{equation*}
\text{if}\quad \text{$\partial B_{r}(0)\subset S$}\quad\text{for some $r > 0$,}\quad\text{then}\quad U_{S}^{\circ} = U^{\circ}\quad\text{and}\quad D_{S} = D.
\end{equation*}
More generally, let $r\frac{h}{|h|}\in S$ for all $h\in S\setminus\{0\}$ and some $r > 0$, as in Example~\ref{ex:interior points}. Then the smallest positively homogeneous set $\{t h\,|\,(t,h)\in\R_{+}\times S\}$ that includes $S$ satisfies
\begin{equation*}
B_{r}(0)\cap S = B_{r}(0)\cap \{t h\,|\,(t,h)\in\R_{+}\times S\}
\end{equation*}
and yields the same interior and differentiability notions as $S$.
\end{Remark}

We readily observe that if $\varphi:U\rightarrow Y$ is $S$-differentiable at $x\in U_{S}^{\circ}$, then it must be \emph{$S$-continuous} there in the sense that
\begin{equation*}
\varphi(x + h) = \varphi(x) + o(1)\quad\text{as $h\rightarrow 0$ on $S$.}
\end{equation*}
Moreover, let $\tilde{S}$ be another $0$-star convex set in $X$ with $S\subset\tilde{S}$. Then $U_{\tilde{S}}^{\circ}\subset U_{S}^{\circ}$ and $\tilde{S}$-differentiability of $\varphi$ at $x\in U_{\tilde{S}}^{\circ}$ implies $S$-differentiability at this point and
\begin{equation*}
D_{\tilde{S}}\varphi(x) = D_{S}\varphi(x)\quad\text{on $V$.}
\end{equation*}
In particular, if $\mathrm{span}(S)$ is dense in $X$ and $\varphi$ admits an extension $\tilde{\varphi}$ to an open set $\tilde{U}$ in $X$ with $U\subset\tilde{U}$ that is differentiable at $x\in U$, then $\varphi$ is $S$-differentiable there and
\begin{equation*}
D\tilde{\varphi}(x) = D_{S}\varphi(x).
\end{equation*}
Hence,~\eqref{eq:derivative of an extension} follows from the definition of the $S$-gradient in~\eqref{eq:oriented gradient}. Apart from this essential fact, another motivation for the $S$-derivative is the \emph{oriented directional derivative}.

\begin{Example}\label{ex:oriented directional derivative}
For $(x,h)\in U_{S}^{\circ}\times S$ there is $\delta > 0$ such that $x + th\in U_{S}^{\circ}$ for all $t\in [0,\delta[$. If the map
\begin{equation*}
[0,\delta[\rightarrow Y,\quad t\mapsto\varphi(x + th)
\end{equation*}
is differentiable at $0$, then its derivative there is the \emph{directional derivative} $D_{h}^{+}\varphi(x)$ of $\varphi$ at $x$ \emph{in the positive direction} of $h$. Namely,
\begin{equation*}
D_{h}^{+}\varphi(x) = \lim_{t\downarrow 0} \frac{\varphi(x + th) - \varphi(x)}{t}.
\end{equation*}
Note that $D_{h}^{+}\varphi(x)$ exists if and only if $\varphi$ is $\mathrm{coni}(h)$-differentiable at $x$. In this case, $D_{h}^{+}\varphi(x) = D_{\mathrm{coni}(h)}\varphi(x)(h)$. In particular, $S$-differentiability of $\varphi$ at $x$ yields that
\begin{equation*}
D_{h}^{+}\varphi(x) = D_{S}\varphi(x)(h)\quad\text{for all $h\in S$}
\end{equation*}
if $S$ is positively homogeneous. Hence, the assertions on the G{\^a}teaux derivative in the introduction are direct consequences of the definition of the $S$-gradient.
\end{Example}

A map $\varphi:U\rightarrow Y$ that is $S$-differentiable at each point of $U_{S}^{\circ}$ is \emph{$S$-differentiable} and by the \emph{$S$-derivative} of $\varphi$ we shall mean the map $D_{S}\varphi:U_{S}^{\circ}\rightarrow\mathcal{L}(V,Y)$, $x\mapsto D_{S}\varphi(x)$. Further, $\varphi$ is continuously $S$-differentiable if $D_{S}\varphi$ is $S$-continuous.

For simplicity, $C_{S}^{1}(U)$ denotes the linear space of all $S$-differentiable $\varphi\in C(U)$ for which $D_{S}\varphi$ is continuous. Thus, if $X$ is a Hilbert space and $U$ is $S$-open, then $C_{S}^{1}(U)$ agrees with the linear space in the introduction, by~\eqref{eq:oriented gradient}.

Next, under a value condition for a $0$-star convex set $T$ in $Y$, the \emph{chain rule} extends to the oriented derivative.

\begin{Lemma}\label{le:chain rule}
Let $W\subset Y$ and $Z$ be a Banach space. If $\varphi:U\rightarrow W$ is $S$-differentiable at $x\in U_{S}^{\circ}$, $\varphi(x)\in W_{T}^{\circ}$, $\psi:W\rightarrow Z$ is $T$-differentiable at $\varphi(x)$ and
\begin{equation}\label{eq:chain rule condition}
\varphi(x + h) - \varphi(x) \in T\quad\text{for all $h\in S$ sufficiently close to zero,}
\end{equation}
then $\psi\circ\varphi$ is $S$-differentiable at $x$ and
\begin{equation*}
D_{S}(\psi\circ \varphi)(x) = D_{T}\psi\big(\varphi(x)\big)\circ D_{S}\varphi(x).
\end{equation*}
\end{Lemma}

\begin{Remark}\label{re:chain rule}
If $T$ is a linear space, or less restrictively $x - y \in T$ for all $x,y\in T$, and $\varphi$ takes all its values in $T$, then~\eqref{eq:chain rule condition} is redundant.
\end{Remark}

\begin{Example}\label{ex:product rule}
Lemma~\ref{le:chain rule} gives the \emph{product rule} for the $S$-derivative: If $\varphi,\psi:U\rightarrow\R$ are $S$-differentiable at $x\in U_{S}^{\circ}$, then so is $\varphi\cdot\psi$ and
\begin{equation*}
D_{S}(\varphi\cdot\psi)(x) = \psi(x)\cdot D_{S}\varphi(x) + \varphi(x)\cdot D_{S}\psi(x).
\end{equation*}
\end{Example}

\begin{Example}\label{ex:path}
Let $(x,h)\in U\times S$ satisfy $[x,x+h]\subset U_{S}^{\circ}$, ensuring that $\gamma_{x,h}:[0,1]\rightarrow x + S$, $t\mapsto x + th$ takes all its values in $U_{S}^{\circ}$. If $\varphi$ is $S$-differentiable, then Example~\ref{ex:interior points} and Lemma~\ref{le:chain rule} show that
\begin{equation*}
\varphi\circ\gamma_{x,h}:[0,1]\rightarrow Y,\quad t\mapsto\varphi(x + th)
\end{equation*}
is right-hand differentiable and $(\varphi\circ\gamma_{x,h})_{+}'(t) = D_{S}\varphi(x + th)(h)$ for all $t\in [0,1[$. If in addition $-S\subset S$, then $\varphi\circ\gamma_{x,h}$ is differentiable, as its left-hand derivative exists and
\begin{equation*}
(\varphi\circ\gamma_{x,h})_{-}'(t) = D_{S}\varphi(x + th)(h)
\end{equation*}
for any $t\in ]0,1]$, by another application of Lemma~\ref{le:chain rule}.
\end{Example}

The preceding example entails the \emph{mean value theorem} for the oriented derivative. To this end, we recall that $-S\subset S$ if and only if $S$ is balanced, that is, $tS\subset S$ for all $t\in [-1,1[$.

\begin{Corollary}\label{co:mean value}
Let $S$ be balanced, $\varphi:U\rightarrow Y$ be $S$-differentiable and $(x,h)\in U\times S$ be such that $[x,x+h]\subset U_{S}^{\circ}$. Then
\begin{equation*}
|\varphi(x + h) - \varphi(x)| \leq \sup_{t\in ]0,1[} |D_{S}\varphi(x + th)|\cdot |h|
\end{equation*}
and in the case $Y = \R$ for any $t\in ]0,1]$ there is $s\in ]0,t[$ such that $\varphi(x + th) - \varphi(x)$ $= D_{S}\varphi(x + sh)(h)$. Moreover, if $\varphi$ is continuously $S$-differentiable, then
\begin{equation}\label{eq:mean value}
\varphi(x + h) = \varphi(x) + \int_{0}^{1}D_{S}\varphi(x + t h)(h)\,dt.
\end{equation}
\end{Corollary}

\begin{Remark}
Let $\varphi$ and $D_{S}\varphi$ be continuous on $U_{S}^{\circ}$ . Then~\eqref{eq:mean value} remains valid even if $h\in S$ fails but $h\in\overline{S}$ and one of the following two conditions hold:
\begin{enumerate}[(i)]
\item $x + B_{|h|}(0)\cap S\subset U_{S}^{\circ}$.

\item $U_{V}^{\circ}$ is convex and $[x, x + h]\subset U_{V}^{\circ}$.
\end{enumerate}
Indeed, in either case there is a sequence $(h_{n})_{n\in\N}$ in $S$ that converges to $h$ and satisfies $[x, x + h_{n}]\subset U_{S}^{\circ}$ for all $n\in\N$, which allows us to apply Corollary~\ref{co:mean value}.
\end{Remark}

Finally, we stress the fact that, under a topological condition, the oriented derivative may be viewed as \emph{projection derivative} in the following sense.

\begin{Proposition}\label{pr:projection derivative}
Let $p:X\rightarrow S$ satisfy $p(h) = h$  and $|p(x)| \leq |x|$ for all $h\in S$ and $x\in X$. Then $B_{\delta}(0)\cap S = p(B_{\delta}(0))$ for each $\delta > 0$ and
\begin{equation*}
U_{S}^{\circ} = \{x\in U\,|\,\exists \delta > 0:\, x + p(B_{\delta}(0))\subset U\}.
\end{equation*}
Moreover, a map $\varphi:U\rightarrow Y$ is $S$-differentiable at $x\in U_{S}^{\circ}$ if and only if there exists $L\in\mathcal{L}(V,Y)$, namely, $L = D_{S}\varphi(x)$, such that
\begin{equation*}
\varphi(x + p(h)) = \varphi(x) + L(p(h)) + o(|h|)\quad\text{as $h\rightarrow 0$.}
\end{equation*}
\end{Proposition}

\begin{Remark}\label{re:projection}
Let $|\cdot|$ be induced by an inner product $\langle\cdot,\cdot\rangle$, which turns $X$ into a Hilbert space. Then any non-empty closed convex set $C$ in $X$ admits a unique map $p_{C}:X\rightarrow C$, called the projection onto $C$, such that
\[
\mathrm{dist}(x,C) = |x - p_{C}(x)|\quad\text{for all $x\in X$}.
\]
It satisfies $p_{C}(h) = h$ for all $h\in C$ and $|p_{C}(x) - p_{C}(y)|^{2} \leq \langle x-y,p_{C}(x) - p_{C}(y)\rangle$ for any $x,y\in X$. Thus, if $S$ is closed and convex, then the hypotheses on $p$ hold for $p_{S}$.
\end{Remark}

\begin{Example}
Let $X$ be a Hilbert space and $S$ be a closed linear space, that is, $S = V$. If $\varphi:U\rightarrow Y$ is independent of the coordinates of $x\in U$ in $V$ in the sense that
\begin{equation*}
\varphi(x) = \varphi(y)\quad\text{for all $x,y\in U$ with $p_{V^{\perp}}(x) = p_{V^{\perp}}(y)$,}
\end{equation*}
which is equivalent to $\varphi(x) = \varphi(\hat{x} + p_{V^{\perp}}(x))$ for all $x\in U$ and some, or all, $\hat{x}\in V$ with $\hat{x} + p_{V^{\perp}}(x)\in U$, then $\varphi$ is $V$-differentiable at any $x\in U_{V}^{\circ}$ and
\begin{equation*}
D_{V}\varphi(x) = 0.
\end{equation*}
Thereby, we recall that the orthogonal complement $V^{\perp}$ of $V$ is another closed linear space and its (orthogonal) projection $p_{V^{\perp}}$ is linear and idempotent.
\end{Example}

\subsection{Differentiability along orthogonal sums in Hilbert spaces}\label{se:2.2}

We decompose the oriented derivative along a countably infinite orthogonal sum in a Hilbert space and introduce the oriented gradient, which justifies the description~\eqref{eq:special oriented gradient} and the identity~\eqref{eq:orthogonal decomposition of the gradient}.

Thus, we assume here that the complete norm $|\cdot|$ on the Banach space $X$ is induced by an inner product $\langle\cdot,\cdot\rangle$ and $S$ is the direct sum $\bigoplus_{n\in\N} S_{n}$ of a sequence $(S_{n})_{n\in\N}$ of pairwise orthogonal $0$-star convex sets in $X$.

Then $S$ consists of all $h\in X$ for which there is a sequence $(h_{n})_{n\in\N}$ in $X$, which is necessarily unique, such that $h_{n}\in S_{n}$ for all $n\in\N$ and $(\sum_{i=1}^{n}h_{i})_{n\in\N}$ converges to $h$. We write $V_{n}$ for the closure of $\mathrm{span}(S_{n})$ for all $n\in\N$ and observe that $V  = \bigoplus_{n\in\N} V_{n}$.

Our arguments succeeding Remark~\ref{re:differentiability} show that $U_{S}^{\circ}\subset U_{S_{n}}^{\circ}$ and if $\varphi:U\rightarrow Y$ is $S$-differentiable at $x\in U_{S}^{\circ}$, then it is $S_{n}$-differentiable there and $D_{S}\varphi(x) = D_{S_{n}}\varphi(x)$ on $V_{n}$ for all $n\in\N$. Thus, 
\begin{equation}\label{eq:decomposition of the derivative}
D_{S}\varphi(x)(h) = \sum_{n=1}^{\infty} D_{S_{n}}\varphi(x)(p_{V_{n}}(h))\quad\text{for all $h\in V$,}
\end{equation}
where $p_{V_{n}}$ is the (orthogonal) projection onto $V_{n}$, as in Remark~\ref{re:projection}, and $p_{V_{n}}(h) = h_{n}$ for any $n\in\N$ and $h\in V$. Further, if $U$ is $S$-open, then $U_{S}^{\circ} = U_{S_{n}}^{\circ}$ for all $n\in\N$.

Now that we have considered these facts, we shall assume until the end of this section that $S_{n}$ is actually balanced for all $n\in\N$, which implies the same property for $S$. Then Corollary~\ref{co:mean value} entails an \emph{integral representation}.

\begin{Lemma}\label{le:continuous differentiability}
If $\varphi:U\rightarrow Y$ is continuously $S_{n}$-differentiable for each $n\in\N$, then
\begin{equation*}
\varphi(x + h_{1} + \cdots + h_{n}) - \varphi(x) = \sum_{i=1}^{n} \int_{0}^{1} D_{S_{i}}\varphi(x + h_{1} + \cdots + h_{i} - (1-t) h_{i})(h_{i})\,dt
\end{equation*}
for all $n\in\N$ and $(x,h)\in U\times S$ such that $x + B_{|h|}(0)\cap S\subset \bigcap_{i=1}^{n} U_{S_{i}}^{\circ}$.
\end{Lemma}

\begin{Remark}\label{re:continuity}
If $U$ is $S$-open, then $\lim_{n\uparrow\infty}\varphi(x + h_{1} + \cdots + h_{n}) = \varphi(x + h)$ as soon as $\varphi$ is $S$-continuous at $x + h$.
\end{Remark}

In consequence, a \emph{characterisation of $S$-differentiability} in terms of differentiability along $(S_{n})_{n\in\N}$ follows.

\begin{Proposition}\label{pr:decomposition of the derivative}
Let $U$ be $S$-open. Then $\varphi:U\rightarrow Y$ is continuously $S$-differentiable if and only if it is continuously $S_{n}$-differentiable for any $n\in\N$ and the following holds:
\begin{enumerate}[(i)]
\item $(\sum_{i=1}^{n}D_{S_{i}}\varphi(x)\circ p_{V_{i}})_{n\in\N}$ converges pointwise and $\sum_{n=1}^{\infty}D_{S_{n}}\varphi(x)\circ p_{V_{n}}$ is continuous for each $x\in U$.

\item $\varphi$ and the map $U\rightarrow\mathcal{L}(V,Y)$, $x\mapsto\sum_{n=1}^{\infty}D_{S_{n}}\varphi(x)\circ p_{V_{n}}$ are $S$-continuous.

\item $\sum_{n=1}^{\infty}\int_{0}^{1}(D_{S_{n}}\varphi(x + h_{1} + \cdots + h_{n} - (1-t)h_{n}) - D_{S_{n}}\varphi(x))(h_{n})\,dt = o(|h|)$ as $h\rightarrow 0$ on $S$ for all $x\in U$.
\end{enumerate}
In this case, $D_{S}\varphi(x) = \sum_{n=1}^{\infty}D_{S_{n}}\varphi(x)\circ p_{V_{n}}$ for any $x\in U$.
\end{Proposition}

\begin{Remark}\label{re:decomposition of the derivative}
From the triangle inequality we infer that condition~(iii) is implied by
\begin{equation*}
\limsup_{n\uparrow\infty}\int_{0}^{1}\bigg|\sum_{i=1}^{n}\big(D_{S_{i}}\varphi(x + h_{1} + \cdots + h_{i} - (1-t)h_{i}) - D_{S_{i}}\varphi(x)\big)\circ p_{V_{i}}\bigg|\,dt = o(1)
\end{equation*}
as $h\rightarrow 0$ on $S$ for each $x\in U$.
\end{Remark}

\begin{Example}
Let $n\in\N$ be such that $S_{m + n} = \{0\}$ for all $m\in\N$. Then $S = \bigoplus_{i=1}^{n} S_{i}$, and Proposition~\ref{pr:decomposition of the derivative} shows that $\varphi$ is continuously $S$-differentiable if and only if it is continuously $S_{i}$-differentiable for all $i\in\{1,\dots,n\}$ and the map
\begin{equation*}
U\rightarrow\mathcal{L}(V,Y),\quad x\mapsto\sum_{i=1}^{n}D_{S_{i}}\varphi(x)\circ p_{V_{i}}
\end{equation*}
is $S$-continuous. In this case, $D_{S}\varphi(x) = \sum_{i=1}^{n}D_{S_{i}}\varphi(x)\circ p_{V_{i}}$ for all $x\in U$. Indeed, as $D_{S_{m + n}}\varphi = 0$ for any $m\in\N$, the $S$-continuity of $\varphi$ and condition~(iii) follow in the if-direction from Lemma~\ref{le:continuous differentiability} and Remark~\ref{re:decomposition of the derivative}.
\end{Example}

If $\varphi:U\rightarrow\R$ is differentiable at $x\in U_{S}^{\circ}$, then Riesz representation theorem yields a unique element $\nabla_{S}\varphi(x)$ in $V$, which we call the \emph{$S$-oriented gradient} or \emph{$S$-gradient} of $\varphi$ at $x$, such that
\begin{equation}\label{eq:oriented gradient}
D_{S}\varphi(x)(h) = \langle\nabla_{S}\varphi(x),h\rangle\quad\text{for all $h\in V$}.
\end{equation}
So, $C_{S}^{1}(U)$ is indeed the linear space of all $S$-differentiable $\varphi\in C(U)$ for which $\nabla_{S}\varphi$ is continuous and Proposition~\ref{pr:decomposition of the derivative} gives an explicit relation to the linear space $\bigcap_{n\in\N} C_{S_{n}}^{1}(U)$.

\begin{Corollary}\label{co:decomposition of the derivative}
Let $U$ be $S$-open. Then $C_{S}^{1}(U)$ consists of all $\varphi\in\bigcap_{n\in\N} C_{S_{n}}^{1}(U)$ such that $\sum_{n=1}^{\infty}|\nabla_{S_{n}}\varphi(x)|^{2} < \infty$, locally uniformly in $x\in U$, and
\begin{equation}\label{eq:gradient decomposition condition}
\sum_{n=1}^{\infty} \int_{0}^{1}\langle\nabla_{S_{n}}\varphi(x + h_{1} + \dots + h_{n} - (1-t) h_{n}) - \nabla_{S_{n}}\varphi(x),h_{n}\rangle\,dt = o(|h|)
\end{equation}
as $h\rightarrow 0$ on $S$ for any $x\in U$. Moreover, each $\varphi\in C_{S}^{1}(U)$ satisfies $\nabla_{S}\varphi = \sum_{n=1}^{\infty}\nabla_{S_{n}}\varphi$, locally uniformly.
\end{Corollary}

\begin{Example}\label{ex:decomposition of the derivative 2}
Let $S_{m + n} = \{0\}$ for every $m\in\N$ and some $n\in\N$. Then Remark~\ref{re:decomposition of the derivative} and the preceding corollary entail that $C_{S}^{1}(U)$ $= \bigcap_{i=1}^{n}C_{S_{i}}^{1}(U)$ and~\eqref{eq:orthogonal decomposition of the gradient} holds.
\end{Example}

\subsection{Higher order derivatives, Taylor's formula and local extrema}\label{se:2.3}

We introduce higher order oriented derivatives, generalise Taylor's formula and discuss necessary and sufficient conditions for local extrema based on the oriented derivative.

First, on the Banach space $X$ we characterise the oriented derivative via invertible linear continuous operators that allows for an analysis of the \emph{oriented Hessian}.

\begin{Lemma}\label{le:differentiability values of the gradient}
Let $S= O(R)$ for some $0$-star convex set $R$ in a Banach space $W$ and an invertible $O\in\mathcal{L}(W,X)$. Then a map $\varphi:U\rightarrow Y$ is $S$-differentiable at $x\in U_{S}^{\circ}$ if and only if
\begin{equation*}
\varphi_{O,x}:\{t\in R\,|\, x + O(t)\in U\}\rightarrow Y,\quad t\mapsto \varphi(x + O(t))
\end{equation*}
is $R$-differentiable at $0$. In this case, $D_{S}\varphi(x) = D_{R}\varphi_{O,x}(0)\circ O^{-1}$.
\end{Lemma}

\begin{Remark}\label{re:differentiability values of the gradient}
Let $S=\tilde{O}(\tilde{R})$ for another $0$-star convex set $\tilde{R}$ in a Banach space $\tilde{W}$ and an invertible $\tilde{O}\in\mathcal{L}(\tilde{W},X)$. If $\varphi_{O,x}$ is $R$-differentiable at $0$, then
\begin{equation*}
\text{$\varphi_{\tilde{O},x}$ is $\tilde{R}$-differentiable at $0$}\quad\text{and}\quad  D_{R}\varphi_{O,x}(0)\circ O^{-1} = D_{\tilde{R}}\varphi_{\tilde{O},x}(0)\circ\tilde{O}^{-1}.
\end{equation*}
So, the representation of $D_{S}\varphi(x)$ is independent of the choice of the invertible operator $O\in\mathcal{L}(W,X)$. Moreover, if $X$ is a Hilbert space and $Y =\R$, then
\begin{equation*}
\nabla_{S}\varphi(x) = O(\nabla_{R}\varphi_{O,x}(0))
\end{equation*}
as soon as $\varphi$ is $S$-differentiable at $x$ and $O^{-1}$ is the adjoint of $O$.
\end{Remark}

\begin{Example}\label{ex:differentiability values of the gradient}
Let $X = \R^{d}$ for some $d\in\N$ and $S\neq \{0\}$. Then there are $k\in\{1,\dots,d\}$ and a matrix $O\in\R^{d\times k}$ whose columns form an orthonormal basis of $V$, which yields that $S = O(R)$ for the $0$-star convex set $R:= O'(S)$.
\end{Example}

\emph{Higher order derivatives} can be introduced, by using a well-known recursion principle. For any $n\in\N\setminus\{1\}$ we write $\mathcal{L}^{n}(X,Y)$ for the Banach space of all $Y$-valued multilinear continuous maps on $X^{n}$, equipped with the general operator norm given by
\begin{equation*}
|L|:=\inf\big\{c\geq 0\,|\,\forall x_{1},\dots,x_{n}\in X:\, |L(x_{1},\dots,x_{n})| \leq c |x_{1}|\cdots|x_{n}|\,\big\}.
\end{equation*}
We set $\mathcal{L}^{1}(X,Y):=\mathcal{L}(X,Y)$ and recall that $\psi_{X,n}:\mathcal{L}(X,\mathcal{L}^{n-1}(X,Y))\rightarrow\mathcal{L}^{n}(X,Y)$ defined by $\psi_{X,n}(L)(x_{1},\dots,x_{n})$ $:= L(x_{1})(x_{2},\dots,x_{n})$ is an isometric isomorphism.

Let $\varphi:U\rightarrow Y$ be a map for which $D_{S}^{n-1}\varphi:U_{S}^{\circ}\rightarrow\mathcal{L}^{n-1}(V,Y)$ is already defined. Then $\varphi$ is \emph{$n$-times $S$-differentiable} at $x\in U_{S}^{\circ}$ if $D_{S}^{n-1}\varphi$ is $S$-differentiable there, in which case
\begin{equation*}
D_{S}^{n}\varphi(x):=\psi_{V,n}(D_{S}(D_{S}^{n-1}\varphi)(x))
\end{equation*}
is the \emph{$n$-th order $S$-derivative} of $\varphi$ at $x$. If $\varphi$ is $n$-times $S$-differentiable and $D_{S}^{n}\varphi$ is $S$-continuous, then we refer to $n$-times continuous $S$-differentiability. This leads to an extended version of \emph{Schwarz's lemma}.

\begin{Lemma}\label{le:Schwarz}
Let $S$ be balanced and $\varphi:U\rightarrow Y$ be twice continuously $S$-differentiable. Then $D_{S}^{2}\varphi(x)$ is symmetric for each $x\in U_{S}^{\circ}$.
\end{Lemma}

\begin{Remark}
If $\varphi$ is $n$-times continuously $S$-differentiable for $n\in\N$ with $n\geq 3$, then it follows inductively as in the proof of Corollary~VII.5.3 in~\cite{AmaEsc08} that $D_{S}^{n}\varphi(x)$ is symmetric.
\end{Remark}

\begin{Example}
Let $d\in\N$, $X=\R^{d}$ and $Y=\R$ and $\varphi$ be only twice $S$-differentiable, that is, it is $S$-differentiable and the same holds for $\nabla_{S}\varphi$. Then the map
\begin{equation*}
H_{S}\varphi:U_{S}^{\circ}\rightarrow\R^{d\times d}
\end{equation*}
whose $j$-th column is $\nabla_{S}(\nabla_{S}\varphi)_{j}$ for all $j\in\{1,\dots,d\}$ is the \emph{$S$-Hessian} of $\varphi$ and it admits the subsequent properties:
\begin{enumerate}[(1)]
\item $H_{S}\varphi$ and the projection matrix $P_{V}$ of $V$ commute and the columns of $H_{S}\varphi$ are eigenvectors of $P_{V}$ with eigenvalue $1$. That means,
\begin{equation*}
H_{S}\varphi P_{V} = P_{V} H_{S}\varphi = H_{S}\varphi.
\end{equation*}
This relation follows from Lemma~\ref{le:differentiability values of the gradient}, Remark~\ref{re:differentiability values of the gradient} and Example~\ref{ex:differentiability values of the gradient}, since $P_{V}$ induces the orthogonal projection $p_{V}:\R^{d}\rightarrow V$, $x\mapsto P_{V}x$ onto $V$.

\item The transpose of $H_{S}\varphi$ is the representation matrix of $D_{S}^{2}\varphi$. That is, for the standard inner product $\langle\cdot,\cdot\rangle$ we have
\begin{equation*}
D_{S}^{2}\varphi(\cdot)(h,k) = \langle H_{S}\varphi h, k\rangle\quad\text{for all $h,k\in V$.}
\end{equation*}
This is a consequence of Lemma~\ref{le:chain rule} and the fact that $\mathcal{L}^{2}(V,\R)$ is isomorphic to the linear space of all $A\in\R^{d\times d}$ such that $AP_{V} = P_{V}A = A$.

\item $H_{S}\varphi$ is continuous if and only if $D_{S}^{2}\varphi$ is. In this case, $H_{S}\varphi$ takes all its values in the linear space of all symmetric $A\in\R^{d\times d}$ satisfying $P_{V} A = A$, by Lemma~\ref{le:Schwarz}.
\end{enumerate}
\end{Example}

Now we are in a position to deduce the \emph{Taylor formula} for the oriented derivative.

\begin{Proposition}\label{pr:Taylor formula}
Let $S$ be balanced, $n\in\N$ and $\varphi:U\rightarrow Y$ be $n$-times continuously $S$-differentiable. Then
\begin{equation}\label{eq:Taylor formula}
\varphi(x + h) - \varphi(x) = \sum_{k=1}^{n}\frac{1}{k!} D_{S}^{k}\varphi(x)(h,\dots,h) + R_{\varphi,n}(x,h)
\end{equation}
for all $(x,h)\in U\times S$ with $[x,x+h]\subset U_{S}^{\circ}$ and the $Y$-valued map $R_{\varphi,n}$ given by
\begin{equation*}
R_{\varphi,n}(x,h) := \int_{0}^{1}\frac{(1-t)^{n-1}}{(n-1)!}\big(D_{S}^{n}\varphi(x + th)  - D_{S}^{n}\varphi(x)\big)(h,\dots,h)\,dt.
\end{equation*}
\end{Proposition}

Finally, let us consider \emph{local $S$-extrema} of a function $\varphi:U\rightarrow\R$. Namely, $\varphi$ has a local $S$-maximum at $x\in U$ if there is $\delta > 0$ such that
\begin{equation*}
\varphi(x +h) \leq \varphi(x)\quad\text{for all $h\in B_{\delta}(0)\cap S$ with $h\neq 0$ and $x + h\in U$.}
\end{equation*}
If the inequality is strict, then such a maximum is called \emph{strict}, and to define (strict) local $S$-minima we replace $\varphi$ by $-\varphi$. Then the oriented directional derivative in Example~\ref{ex:oriented directional derivative} gives a \emph{necessary criterion} for local $S$-extrema to be attained in $U_{S}^{\circ}$.

\begin{Lemma}\label{le:local extrema}
Let $\varphi:U\rightarrow\R$ have a local $S$-maximum at $x\in U_{S}^{\circ}$. If $D_{h}^{+}\varphi(x)$ exists for $h\in S$, then $D_{h}^{+}\varphi(x)\leq 0$. In particular,
\begin{equation*}
D_{S}\varphi(x)(h)\leq 0\quad\text{for any $h$ in the closure of $\mathrm{coni}(S)$}
\end{equation*}
as soon as $S$ is positively homogeneous and $\varphi$ is $S$-differentiable at $x$.
\end{Lemma}

For \emph{sufficient conditions} that determine whether a function has a local $S$-extremum in $U_{S}^{\circ}$, let us call a symmetric operator $L\in\mathcal{L}^{2}(V,\R)$ negative \emph{semidefinite on $S$} if $L(h,h) \leq 0$ for all $h\in S\setminus\{0\}$. If this inequality is strict, then we refer to definiteness.

Further, $L$ is strongly negative definite on $S$ if there is $\lambda > 0$ such that $L(h,h) \leq - \lambda|h|^{2}$ for any $h\in S$. Clearly, if $X$ is a Hilbert space of finite dimension and $S$ is a linear space, then the two preceding notions agree.

To introduce semidefiniteness and (strong) definiteness in the positive sense, we take $-L$ instead of $L$, and $L$ is indefinite on $S$ if it is neither negative nor positive semidefinite on the same set.

\begin{Lemma}\label{le:local extrema 2}
For $\varphi:U\rightarrow\R$ and $x\in U_{S}^{\circ}$ the following two assertions hold:
\begin{enumerate}[(i)]
\item If there is a symmetric $L\in\mathcal{L}^{2}(V,\R)$ that is strongly negative definite on $S$ and
\begin{equation*}
\varphi(x + h) \leq \varphi(x) + \frac{1}{2}L(h,h) + o(h)\quad\text{as $h\rightarrow 0$ on $S$,}
\end{equation*}
then $\varphi$ attains a strict local $S$-maximum at $x$.

\item If there is a symmetric $L\in\mathcal{L}^{2}(V,\R)$ that is indefinite on $S$ such that
\begin{equation*}
\varphi(x + h) = \varphi(x) + \frac{1}{2}L(h,h) + o(h)\quad\text{as $h\rightarrow 0$ on $S$,}
\end{equation*}
then $\varphi$ does not have a local $S$-extremum at $x$.
\end{enumerate}
\end{Lemma}

\begin{Remark}
Assume that $S$ is balanced, $\varphi$ is twice continuously $S$-differentiable and $D_{S}\varphi(x) = 0$. Then Proposition~\ref{pr:Taylor formula} and Lemma~\ref{le:local extrema 2} yield two implications:
\begin{enumerate}[(i)]
\item If $D_{S}^{2}\varphi(x)$ is strongly negative definite on $S$, then $\varphi$ has a strict local $S$-maximum at $x$.

\item If $D_{S}^{2}\varphi(x)$ is indefinite on $S$, then $\varphi$ cannot have a local $S$-extremum at $x$.
\end{enumerate}
\end{Remark}

\section{Proofs of all the results}\label{se:3}

\begin{proof}[Proof of Lemma~\ref{le:chain rule}]
By~\eqref{eq:chain rule condition}, there are $\delta,\varepsilon > 0$ and functions $r:B_{\delta}(0)\cap S\rightarrow Y$, $s:B_{\varepsilon}(0)\cap T\rightarrow Z$ such that $x + h\in U$, $\varphi(x) + l\in W$,
\begin{equation*}
\varphi(x + h) - \varphi(x)\in B_{\varepsilon}(0)\cap T,\quad \varphi(x + h) = \varphi(x) + D_{S}\varphi(x)(h) + r(h)
\end{equation*}
and $\psi(\varphi(x) + l) = \psi(\varphi(x)) + D_{T}\psi(\varphi(x))(l) + s(l)$ for any $h\in B_{\delta}(0)\cap S$ and $l\in B_{\varepsilon}(0)\cap T$. Further, $r(h) = o(|h|)$ as $h\rightarrow 0$ and $s(l) = o(|l|)$ as $l\rightarrow 0$. Thus,
\begin{equation*}
(\psi\circ\varphi)(x + h) = (\psi\circ\varphi)(x) + D_{T}\psi(\varphi(x))D_{S}\varphi(x)h + t(h)
\end{equation*}
for each $h\in B_{\delta}(0)\cap S$, where the map $t:B_{\delta}(0)\cap S\rightarrow Z$, $h\mapsto D_{T}\psi(\varphi(x))r(h)$ $ +\,  s(D_{S}\varphi(x)h + r(h))$ satisfies $t(h) = o(|h|)$ as $h\rightarrow 0$.
\end{proof}

\begin{proof}[Proof of Corollary~\ref{co:mean value}]
By Example~\ref{ex:path}, the composition $\varphi\circ\gamma_{x,h}$ is differentiable and admits $D_{S}\varphi(x + th)(h)$ as derivative at any $t\in [0,1]$. Thus, the asserted bound and identity follow from $\varphi(x + th) - \varphi(x) = (\varphi\circ\gamma_{x,h})(t) - (\varphi\circ\gamma_{x,h})(0)$ and the standard mean value theorem, which can be found in~\cite[Theorems IV.2.4~and~IV.2.18]{AmaEsc05}.

If $\varphi$ is continuously $S$-differentiable, then $\varphi\circ\gamma_{x,h}$ is continuously differentiable and the Fundamental Theorem of Calculus, stated in this generality in~\cite[Corollary~VI.4.14]{AmaEsc08}, yields the second claimed identity.
\end{proof}

\begin{proof}[Proof of Proposition~\ref{pr:projection derivative}]
The representation of $U_{S}^{\circ}$ is a direct consequence of the first claim and, as $p(h)=h$ for all $h\in S$, we have $B_{\delta}(0)\cap S$ $\subset p(B_{\delta}(0))$ for $\delta > 0$. The converse inclusion follows from $p(X)\subset S$ and $|p(x)|\leq |x|$ for all $x\in X$.

Next, if $\varphi$ is $S$-differentiable at $x$, then there are $\delta > 0$ and $r:B_{\delta}(0)\cap S\rightarrow Y$ such that $x + h\in U$, $\varphi(x + h) = \varphi(x) + D_{S}\varphi(x)(h) + r(h)$ for all $h\in B_{\delta}(0)\cap S$ and $r(h) = o(|h|)$ as $h\rightarrow 0$. Hence,
\begin{equation}\label{eq:projection derivative}
x + p(h)\in U\quad\text{and}\quad\varphi(x + p(h)) = \varphi(x) + L(p(h)) + s(h)\quad\text{for all $h\in B_{\delta}(0)$},
\end{equation}
where $L:=D_{S}\varphi(x)$ and $s:B_{\delta}(0)\rightarrow Y$, $h\mapsto r(p(h))$ satisfies $s(h) = o(|h|)$ as $h\rightarrow 0$. Conversely, if there is $L\in\mathcal{L}(V,Y)$ such that~\eqref{eq:projection derivative} holds for some $\delta > 0$ and $s:B_{\delta}(0)\rightarrow Y$ satisfying $s(h) = o(|h|)$ as $h\rightarrow 0$,  then we restrict $s$ to $B_{\delta}(0)\cap S$ to see that $\varphi$ must be $S$-differentiable at $x$.
\end{proof}

\begin{proof}[Proof of Lemma~\ref{le:continuous differentiability}]
The sequence $(x_{j})_{j\in\{0,\dots,n\}}$ in $x + B_{|h|}(0)\cap S$ given by $x_{0}:=x$ and $x_{j}:=x + \sum_{i=1}^{j}h_{i}$ for all $j\in\{1,\dots,n\}$ satisfies $\varphi(x_{n}) = \varphi(x) + \sum_{i=1}^{n}\varphi(x_{i}) - \varphi(x_{i-1})$. 

Further, $\varphi(x_{i}) - \varphi(x_{i-1}) = \int_{0}^{1}D_{S_{i}}\varphi(x_{i-1} + th_{i})(h_{i})\,dt$ for each $i\in\{1,\dots,n\}$, due to Corollary~\ref{co:mean value}, as $|x_{i-1} + th_{i} - x| \leq |h|$ for all $t\in [0,1]$. So, the asserted identity holds.
\end{proof}

\begin{proof}[Proof of Proposition~\ref{pr:decomposition of the derivative}]
By our derivation of~\eqref{eq:decomposition of the derivative}, if $\varphi$ is continuously $S$-differentiable, then continuous $S_{n}$-differentiability follows for all $n\in\N$ and~(i) and~(ii) hold.

For this reason, let $\varphi$ be solely $S$-continuous and continuously $S_{n}$-differentiable for any $n\in\N$, and $x\in U$ and $\delta > 0$ be such that $x + B_{\delta}(0)\cap S\subset U$. Then Lemma~\ref{le:continuous differentiability} yields that
\begin{align*}
\varphi(x + h_{1} + \cdots + h_{n}) &= \varphi(x) + \sum_{i=1}^{n}D_{S_{i}}\varphi(x)(h_{i})\\
&\quad + \sum_{i=1}^{n}\int_{0}^{1} \big(D_{S_{i}}\varphi(x + h_{1} + \cdots + h_{i} - (1-t)h_{i}) - D_{S_{i}}\varphi(x)\big)(h_{i})\,dt
\end{align*}
for all $n\in\N$ and $h\in B_{\delta}(0)\cap S$. Hence, from Remark~\ref{re:continuity} we infer that $\varphi$ is $S$-differentiable at $x$ if and only if $(\sum_{i=1}^{n}D_{S_{i}}\varphi(x)\circ p_{V_{i}})_{n\in\N}$ converges pointwise, $\sum_{n=1}^{\infty}D_{S_{n}}\varphi(x)\circ p_{V_{n}}$ is continuous and the limit in~(iii) holds at $x$, which gives the claim.
\end{proof}

\begin{proof}[Proof of Corollary~\ref{co:decomposition of the derivative}]
Proposition~\ref{pr:decomposition of the derivative} shows us that $C_{S}^{1}(U)\subset \bigcap_{n\in\N} C_{S_{n}}^{1}(U)$ and for any $\varphi\in C_{S}^{1}(U)$ we have $\langle\nabla_{S}\varphi,h\rangle = \sum_{n=1}^{\infty}\langle\nabla_{S_{n}}\varphi,h_{n}\rangle$ for each $h\in V$. This entails that
\begin{equation*}
\nabla_{S}\varphi = \sum_{n=1}^{\infty}\nabla_{S_{n}}\varphi\quad\text{and hence,}\quad |\nabla_{S}\varphi|^{2} = \sum_{n=1}^{\infty} |\nabla_{S_{n}}\varphi|^{2}.
\end{equation*}
As the continuity of $\nabla_{S}\varphi$ implies that $(\sum_{i=1}^{n}\nabla_{S_{i}}\varphi)_{n\in\N}$ is a local uniform Cauchy sequence, $\nabla_{S}\varphi$ is not only the pointwise but in fact the local uniform limit.

Conversely, let $\varphi\in\bigcap_{n\in\N} C_{S_{n}}^{1}(U)$ satisfy $\sum_{n=1}^{\infty} |\nabla_{S_{n}}\varphi(x)|^{2} < \infty$, locally uniformly in $x\in U$. Then $(\sum_{i=1}^{n}\nabla_{S_{i}}\varphi)_{n\in\N}$ is locally uniformly Cauchy, the local uniform limit $\sum_{n=1}^{\infty}\nabla_{S_{n}}\varphi$ is continuous and
\begin{equation*}
\sum_{n=1}^{\infty}|\langle\nabla_{S_{n}}\varphi(x),h_{n}\rangle| \leq \bigg(\sum_{n=1}^{\infty} |\nabla_{S_{n}}\varphi(x)|^{2}\bigg)^{\frac{1}{2}} \cdot |h|\quad\text{for all $h\in V$},
\end{equation*}
by the Cauchy-Schwarz inequality. Thus, conditions (i) and (ii) of Proposition~\ref{pr:decomposition of the derivative} hold and the claim follows, since~\eqref{eq:gradient decomposition condition} is condition~(iii) in terms of the gradient.
\end{proof}

\begin{proof}[Proof of Lemma~\ref{le:differentiability values of the gradient}]
For only if we take $\delta > 0$ and $r:B_{\delta}(0)\cap S\rightarrow Y$ such that 
\begin{equation}\label{eq:differentiability values of the gradient}
x + h\in U\quad\text{and}\quad \varphi(x + h) = \varphi(x) + K(h) + r(h)\quad\text{for any $h\in B_{\delta}(0)\cap S$}
\end{equation}
with $K := D_{S}\varphi(x)$ and $r(h) = o(|h|)$ as $h\rightarrow 0$. By setting $\varepsilon := (1 + |O|)^{-1}\delta$ and $L:= D_{S}\varphi(x)\circ O$, we obtain that
\begin{equation}\label{eq:differentiability values of the gradient 2}
x + O(t)\in U\quad\text{and}\quad \varphi_{O,x}(t) = \varphi_{O,x}(0) + L(t) + s(t)\quad\text{for all $t\in B_{\varepsilon}(0)\cap R$,}
\end{equation}
where $s:B_{\varepsilon}(0)\cap R\rightarrow Y$, $t\mapsto r(O(t))$ satisfies $s(t) = o(|t|)$ as $t\rightarrow 0$.

For if let $\varepsilon > 0$ and $s:B_{\varepsilon}(0)\cap R\rightarrow Y$ be such that~\eqref{eq:differentiability values of the gradient 2} holds for $L = D_{R}\varphi_{O,x}(0)$ and some $s:B_{\varepsilon}(0)\cap R\rightarrow Y$ such that $s(t) = o(|t|)$ as $t\rightarrow 0$. Then~\eqref{eq:differentiability values of the gradient} follows for
\begin{equation*}
K = D_{R}\varphi_{O,x}(0)\circ O^{-1},\quad \delta = (1 + |O^{-1}|)^{-1}\varepsilon
\end{equation*}
and $r:B_{\delta}(0)\cap S\rightarrow Y$ given by $r(h):=s(O^{-1}(h))$. Since $r(h) = o(|h|)$ as $h\rightarrow 0$, our verification is complete.
\end{proof}

\begin{proof}[Proof of Lemma~\ref{le:Schwarz}]
By the bilinearity and continuity of $D_{S}^{2}\varphi(x)$, it suffices to show that $D_{S}^{2}\varphi(x)(h,k) = D_{S}^{2}\varphi(x)(k,h)$ for any $h,k\in S$ with $|h|\vee |k| < \delta$ for some $\delta > 0$.

We take $\delta > 0$ such that $x + B_{\delta}(0)\cap S\subset U$ and let $h,k\in B_{\delta/2}(0)\cap S$. Then $x + sh + tk\in U_{S}^{\circ}$ for all $s,t\in [0,1]$ and two applications of Corollary~\ref{co:mean value} yield that
\begin{equation*}
\varphi(x + h + k) - \varphi(x + k) - \varphi(x + h) + \varphi(x) = \int_{0}^{1}\int_{0}^{1}D_{S}^{2}\varphi(x + sh + tk)(h,k)\,ds\,dt.
\end{equation*}
The symmetry of the left-hand terms allows us to interchange the variables $h$ and $k$ and see that these expressions agree with $\int_{0}^{1}\int_{0}^{1}D_{S}^{2}\varphi(x + sk + th)(k,h)\,ds\,dt$. Hence,
\begin{equation*}
|D_{S}^{2}\varphi(x)(h,k) - D_{S}^{2}\varphi(x)(k,h)| \leq 2\sup_{s,t\in ]0,1[} |D_{S}^{2}\varphi(x + sh + tk) - D_{S}^{2}\varphi(x)|\cdot |h|\cdot |k|.
\end{equation*}
Finally, we replace $h$ and $k$ in this inequality by $u h$ and $u k$ for some $u\in ]0,1]$ and infer from the $S$-continuity of $D_{S}^{2}\varphi$ that $D_{S}^{2}\varphi(x)(h,k) = D_{S}^{2}\varphi(x)(k,h)$.
\end{proof}

\begin{proof}[Proof of Proposition~\ref{pr:Taylor formula}]
We show the claim by induction over $n\in\N$. As the case $n=1$ is covered by Corollary~\ref{co:mean value}, we may assume that the assertion holds for some $n\in\N$.

Let $\varphi:U\rightarrow Y$ be $(n+1)$-times continuously $S$-differentiable and $(x,h)\in U\times S$ satisfy $[x,x+h]\subset U_{S}^{\circ}$. Then Example~\ref{ex:path} entails that the map
\begin{equation*}
D_{S}^{n}\varphi(\cdot)(h,\dots,h)\circ\gamma_{x,h}:[0,1]\rightarrow Y,\quad t\mapsto D_{S}^{n}\varphi(x + th)(h,\dots,h)
\end{equation*}
is continuously differentiable and $(D_{S}^{n}\varphi(\cdot)(h,\dots,h)\circ\gamma_{x,h})'(t) = D_{S}^{n+1}\varphi(x + th)(h,\dots,h)$ for all $t\in [0,1]$. Hence, integration by parts yields that
\begin{align*}
R_{\varphi,n}(x,h) &= \int_{0}^{1}\frac{(1-t)^{n}}{n!}D_{S}^{n+1}\varphi(x + th)(h,\dots,h)\,dt\\
&= \frac{1}{(n+1)!}D_{S}^{n+1}\varphi(x)(h,\dots,h) + R_{\varphi,n+1}(x,h).
\end{align*}
Due to the induction hypothesis, this shows that~\eqref{eq:Taylor formula} holds when $n$ is replaced by $n+1$.
\end{proof}

\begin{proof}[Proof of Lemma~\ref{le:local extrema}]
The second claim follows from the first, since in Example~\ref{ex:oriented directional derivative} we verified that $D_{h}^{+}\varphi(x) = D_{S}\varphi(x)(h)$ for any $h\in S$ in case that $S$ is positively homogeneous and $\varphi$ is $S$-differentiable at $x$.

To show the first assertion, we choose $\delta > 0$ such that $x + th\in U_{S}^{\circ}$ for all $t\in [0,\delta[$. Then $f:[0,\delta[\rightarrow\R$, $t\mapsto\varphi(x + th)$ is right-hand differentiable at $0$ and admits a local maximum there. Hence, $D_{h}^{+}\varphi(x) = f_{+}'(0) = \lim_{t\downarrow 0} (f(t) - f(0))/t \leq 0$.
\end{proof}

\begin{proof}[Proof of Lemma~\ref{le:local extrema 2}]
(i) We take $\lambda > 0$ satisfying $L(h,h) \leq -\lambda |h|^{2}$ for all $h\in S$ and let $\alpha\in ]0,\lambda/2[$. Then there is some $\delta > 0$ such that $x + h\in U_{S}^{\circ}$ and $\varphi(x + h) - \varphi(x)$ $\leq \frac{1}{2}L(h,h) + (\frac{\lambda}{2} - \alpha) |h|^{2} \leq - \alpha |h|^{2}$ for each $h\in B_{\delta}(0)\cap S$.

(ii) By hypothesis, we find $h_{+},h_{-}\in S\setminus\{0\}$ such that $\alpha_{+} := L(h_{+},h_{+})|h_{+}|^{-2}$ and $\alpha_{-} := -L(h_{-},h_{-})|h_{-}|^{-2}$ are positive. Then, for $\beta\in ]0,\alpha_{+}\wedge\alpha_{-}[$ we choose $\delta > 0$ and $r:B_{\delta}(0)\cap S\rightarrow\R$ such that $x +h\in U$,
\begin{equation*}
\varphi(x + h) - \varphi(x) = \frac{1}{2}L(h,h) + r(h)\quad\text{and}\quad |r(h)| \leq \frac{\beta}{2}|h|^{2}
\end{equation*}
for all $h\in B_{\delta}(0)\cap S$. So, let $\varepsilon > 0$ satisfy $t h_{+},th_{-}\in B_{\delta}(0)$ for any fixed $t\in ]0,\varepsilon[$. Then $\varphi(x + th_{+}) - \varphi(x) \geq \frac{t^{2}}{2}(\alpha_{+}-\beta)|h_{+}|^{2}$ and $\varphi(x + th_{-}) - \varphi(x) \leq - \frac{t^{2}}{2}(\alpha_{-} - \beta)|h_{-}|^{2}$.
\end{proof}

\end{document}